\documentclass[12pt,reqno]{amsart}
\usepackage{amsmath,amscd,amsfonts,amssymb, color,bbm, bm,
braket, xcolor, setspace, cancel}
\usepackage{latexsym, graphicx, pstricks,rotating, enumerate}
\usepackage[pdftex,bookmarks,colorlinks]{hyperref}
\definecolor{darkblue}{rgb}{0.0,0.0,0.3}
\hypersetup{colorlinks=true,citecolor=darkblue,
unicode=true,pdftitle=symplectic~~szeg\H o,pdfauthor=Ritabrata
Sengupta,bookmarks=false, urlcolor=darkblue} 
\numberwithin{equation}{section}
\newtheorem{thm}{Theorem}[section]
\newtheorem{prop}[thm]{Proposition}
\newtheorem{defin}{Definition}[section]
\newtheorem{rem}{Remark}[section]

\newtheorem{lemma}[thm]{Lemma}
\newtheorem{cor}[thm]{Corollary}

\newcommand{\tr}{\mathrm{Tr\,}}
\newcommand{\dd}{\mathrm{\;d}}

\def \e   {\text {\rm e}}

\newcommand{\wt}{\widetilde}
\begin{document}
\title{A Szeg\H o type theorem and distribution of symplectic
eigenvalues}
\author{Rajendra Bhatia}
\address{Department of Mathematics, Ashoka
University, Rajiv Gandhi Education City, P.O.Rai, Sonepat, Haryana 131 029, India.} 
\email[Rajendra Bhatia]{\href{mailto:rajendra.bhatia@ashoka.edu.in}{rajendra.bhatia@ashoka.edu.in}}
\author{Tanvi Jain}
\address{Theoretical Statistics and Mathematics Unit, Indian
Statistical Institute, Delhi Centre, 7 S J S Sansanwal Marg,
New Delhi 110 016, India.}
\email[Tanvi Jain]{\href{mailto:tanvi@isid.ac.in}{tanvi@isid.ac.in}}
\author{Ritabrata Sengupta}
\address{Department of Mathematical Science, Indian Institute of
Science, Education, \& Research (IISER), Berhampur, Transit campus -
Govt. ITI, NH 59, Berhampur 760 010, Ganjam, Odisha, India.}
\email[Ritabrata
Sengupta]{\href{mailto:rb@iiserbpr.ac.in}{rb@iiserbpr.ac.in}}

\begin{abstract}
\par We study the properties of stationary
G-chains in terms of their generating functions. In
particular, we prove an analogue of the Szeg\H o limit theorem for symplectic eigenvalues,
derive an expression for the entropy rate of stationary quantum Gaussian processes,
and study the distribution of symplectic eigenvalues of truncated block Toeplitz matrices.
We also introduce a concept of symplectic numerical range, analogous to that of numerical range,
and study some of its basic properties, mainly in the context of block Toeplitz operators.


\smallskip

\end{abstract}
\thanks{The work of RB is supported by a Bhatnagar Fellowship of the CSIR. 
TJ acknowledges financial support from {\sf SERB
MATRICS grant number MTR/2018/000554}. RS acknowledges
financial support from {\sf SERB MATRICS grant number
MTR/2017/000431} and {\sf DST/ICPS/QuST/Theme-2/2019/General Project number Q-90}.
} 
\maketitle

\maketitle
\section{Introduction}
\par A quantum  state $\rho$  in a bosonic
Fock space $\Gamma(\mathbb{C}^k)$ is a positive semidefinite
operator with trace one. Let  $q_1, p_1;
\ldots;q_k, p_k$ be $k$ pairs of position-momentum 
observables of a quantum system with $k$ degrees of
freedom satisfying the canonical commutation relations. We
introduce the observables $(X_1, X_2, \ldots, X_{2k-1},
X_{2k}) = (q_1, p_1, \ldots, q_k, p_k)$. Then if $\rho$ has
finite second moments, we write the covariance
matrix of $\rho$ as $A=[\mathrm{Cov}_\rho (X_i, X_j) ]_{i,j=1}^{2k},$ where 
\[\mathrm{Cov}_\rho (X_i,X_j)
= \tr \frac{1}{2} (X_1X_2+X_2X_1)\rho -(\tr X_1\rho)(\tr X_2
\rho).\]
The complete Heisenberg uncertainty principle for
all the position and momentum observables assumes the form of
the following matrix inequality: 
\begin{equation}\label{eqb1}
A+\frac{\imath}{2}J_{2k} \ge0,
\end{equation}
where $J_{2k}= {\underset{k\,
\textrm{times}}{\underbrace{J_2\oplus\cdots\oplus J_2}}}$
with $J_2=\begin{bmatrix}0 & 1\\ -1 & 0\end{bmatrix}.$ 

\par Following the terminology in \cite{krprb2}, we call 
a real $2k \times 2k$ positive definite matrix $A$ satisfying inequality
\eqref{eqb1} a \emph{G-matrix}. 
A standard result in quantum theory states that a $k$-mode,
mean zero, Gaussian quantum state  is uniquely represented by
its covariance matrix, which is a G-matrix. Conversely any
$2k \times 2k$ G-matrix is the
covariance matrix of a unique (up to permutation) $k$-mode mean zero quantum
Gaussian state in the Fock space $\Gamma(\mathbb{C}^k)$;  \cite{krpg,
MR2986302}. Finite mode quantum Gaussian states and quantum
Gaussian processes have been extensively studied in quantum
optics, quantum probability, and quantum information -  both
in theory as well as in experiments.  A comprehensive 
survey of Gaussian states and their properties can be found in the two 
books of Holevo \cite{holevo1, MR2986302}. For their applications to 
quantum information theory the reader is referred to the
survey article by Weedbrook et al \cite{RevModPhys.84.621},
Holevo's book \cite{MR2986302}, and the new book of
Serafini \cite{Sarafini}.

\par In the present paper, our concern is with a stationary
quantum Gaussian process. This is a chain of
finite mode  ($k$ mode) quantum Gaussian states exhibiting
stationarity. Let $\{\rho_n\}$ be a chain of
quantum Gaussian states with covariance matrices $\{T_n\}$.
The stationarity property means that each $T_n$ is a
positive definite block Toeplitz matrix such that $T_n$ is
the leading principal sub-matrix of $T_{n+1}$. This
sequence $\{T_n\}$ gives rise to an infinite block Toeplitz
matrix $\Sigma$. We call this chain $\{\rho_n\}$ of quantum
Gaussian states a \emph{stationary} quantum Gaussian process
and the infinite matrix $\Sigma$  a \emph{G-chain}
\cite{krprb2}. Thus a G-chain $\Sigma$ is an infinite block Toeplitz matrix.
The classical
version of such objects has been well studied in
probability theory. (See for instance \cite{MR543837}.) A study
of the quantum version has been initiated in
\cite{krprb2, krprb4}. In order to study G-chains, we need to study properties of infinite block Toeplitz matrices with blocks of size $2k \times 2k$. 
Every leading $n \times n$ principal block sub-matrix gives a
covariance matrix of an $nk$-mode quantum Gaussian state. 
Toeplitz matrices play an
important part in the study of stationary processes in
classical probability theory as well. See, e.g., 
Grenander and Szeg\H o \cite{grenander-szego}.

\par  Among real positive definite matrices, G-matrices are characterised by a simple property of their symplectic eigenvalues. 
{\it Williamson's theorem} \cite{zbMATH03035616} tells us that
for every $2k\times 2k$ real positive definite matrix $A,$
there exists a symplectic matrix $M$ such that
\[MAM^T= d_1(A) I_2 \oplus \cdots \oplus d_k(A) I_2,\]
where $d_1(A)\le \cdots\le d_k(A)$ are positive numbers uniquely determined by $A$.
These are uniquely determined by $A.$
We call these numbers the {\it symplectic eigenvalues}
of $A.$ 
 We can see that a matrix $A$ is a
G-matrix if and only if all its symplectic eigenvalues
$d_j(A)\ge \frac{1}{2}.$  There has recently been 
considerable interest in the study of various properties of
symplectic eigenvalues (see for instance
\cite{bhatia-jain,  gosson, MR2391195, PhysRevA.73.012330}), due to
their close connection with quantum optics and
thermodynamics \cite{MR2986302, mandel_wolf}. 
\par Given a $k$-mode quantum Gaussian state with covariance
matrix $A$, the {\it von Neumann entropy} of the state is
given by
\begin{equation}\label{eqn1}
S(A)= \sum_{j=1}^k \frac{2d_j(A)+1}{2} H\left( \frac{2d_j(A)
-1}{2d_j(A)+1} \right),
\end{equation}
where $H$ is the Shannon entropy function given by $H(t) =
-t \log t - (1-t) \log (1-t), \, 0 \le t \le 1,$ and $H(0)
=H(1) =0.$ See \cite{MR3800794, krp9}, or \cite{Sarafini} pages 61 -- 62.  
Let $T_n$ be the covariance matrix of a
$k$-mode stationary quantum Gaussian process, truncated at
level $n.$ The \emph{entropy rate} of the process is defined as
\[\lim_{n \to \infty} \frac{S(T_n)}{n}.\]
An important problem in information theory has been the 
study of the entropy rate of any given stationary process. This
can be very complicated \cite{coverthomas, gray-book}.
The entropy rate for a certain type of stationary quantum Gaussian
process was
calculated in \cite{krprb2}.
We compute the entropy rate for a more general class, namely, the class of bounded partially symmetric stationary quantum Gaussian processes.
Let $\Sigma=\begin{bmatrix}A_{i-j}\end{bmatrix}$ be a G-chain corresponding to a stationary quantum Gaussian process.
We 
 call this process {\it bounded} if $\Sigma$ is a bounded operator on $l^2_{2k}$ (the space of square summable sequences of elements of $\mathbb{C}^{2k}$).
In this case $\Sigma$ is a Toeplitz operator generated by a matrix symbol $\wt{A}$ in $L^\infty_{2k\times 2k}.$
The process is {\it partially symmetric} if $A_{-n}=A_n$ for all $n\in\mathbb{N}.$
We show that a stationary quantum Gaussian process is partially symmetric and bounded if and only if
its corresponding G-chain is generated by an $\wt{A}$ in $L^\infty_{2k\times 2k}$ such that $\wt{A}(\theta)$ is a G-matrix for almost all $\theta.$
The computation of the entropy rate requires a study of the
distribution of symplectic eigenvalues of block Toeplitz
matrices. To achieve this  we 
prove a symplectic analogue of a fundamental
theorem for the distribution of eigenvalues of
Toeplitz matrices, well-known as the Szeg\H o limit theorem
\cite{grenander-szego, MR0059482}. The classical Szeg\H o theorem can
be stated as follows: \emph{ Suppose $\varphi: (-\pi, \pi)
\to \mathbb{R}$ is an essentially bounded function, and
$(T_n)$ is
the sequence of Hermitian Toeplitz matrices generated by
$\varphi.$
Then for every function $f,$ continuous on the interval}
$[\textrm{essinf}~ \varphi,\, \textrm{esssup} ~\varphi],$ \emph{one has
\begin{equation}\label{eqn2}
\lim_{n \to \infty} \frac{1}{n}
\sum_{j=1}^n f(\lambda_j(T_n)) = \frac{1}{2\pi}
\int_{-\pi}^\pi f(\varphi(x)) \dd x,
\end{equation}
where $\lambda_j(T_n), \, j = 1, 2, \cdots, n,$ are the
eigenvalues of $T_n.$ }  Many different versions and  proofs
of this theorem are available in the
literature \cite{bs, MR0270218,  simon-ortho-1, tilli-1, tilli-2, MR761763,
Tyrtyshnikov, MR0409511, MR0409512}.
We prove an analogue of this
theorem for symplectic eigenvalues,
and apply this to compute the entropy rate and to 
study the distribution of symplectic eigenvalues of block Toeplitz matrices.
In particular we prove that the union of the set of all
symplectic eigenvalues of truncated $n \times n$ block
Toeplitz matrices $T_n(\wt{A})$ is dense in the set of all
symplectic eigenvalues of $\wt{A}(\theta)$ where
 $\wt{A}(\theta)$   varies over  the essential range of $\wt{A}.$

\par In classical operator theory, the numerical range is 
an important and useful concept. We introduce an analogous notion
of the symplectic numerical range  and study its basic properties.
We show that the closure of the symplectic numerical range of an operator is convex and contains the symplectic spectrum.
We give a relationship between the symplectic numerical
ranges of truncated block Toeplitz matrices and their
symbol. This, in turn, helps us to have a better understanding
of the distribution of symplectic eigenvalues of the
truncated block Toeplitz matrices.

\par The paper is organised as follows:
We give some basic notations and results in Section
\ref{S2}, introduce the notion of symplectic numerical range in
Section \ref{S3}, and study some of its basic properties,
especially in the context of block Toeplitz operators.
In Section \ref{S4} we prove a symplectic analogue of
Szeg\H o limit theorem, and give its applications.

\section{Preliminaries} \label{S2}

\par We begin with some basic facts about
Toeplitz operators. For proofs and other
details, the reader may refer to the book of B\"ottcher and
Silbermann \cite{bs}.

\par Let $L^\infty_{k\times k}$ denote the set of all
functions $\wt{A}=\begin{bmatrix}\wt{a}_{ij}\end{bmatrix}$ from $[-\pi,\pi]$ to the set of all $k\times k$ complex matrices, with
$\wt{A}(-\pi)=\wt{A}(\pi)$ and $\wt{a}_{ij}$ essentially bounded for all $i,j=1,\ldots,k.$
For an $\wt{A}$
in $L^\infty_{k\times k},$ we define
\begin{equation}\label{eqn3}
\|\wt{A}\|={\underset{\theta\in
[-\pi,\pi]}{\textrm{esssup}}}\|\wt{A}(\theta)\|,
\end{equation}
where $\|\wt{A}(\theta)\|$ denotes the operator norm of
$\wt{A}(\theta).$ It is easy to see that $\|\cdot\|$ is a
norm on $L^\infty_{k\times k}.$
The space $L^\infty_{k\times k}$ is a $C^*$-algebra with the usual operations.
Let $L^2_k$ be the set of all functions
$\wt{x}=(\wt{x}_1,\ldots,\wt{x}_k)$ from $[-\pi,\pi]$ to $\mathbb{C}^k$ with $\wt{x}(-\pi)=\wt{x}(\pi)$ and $\wt{x}\in L^2$ for all $i=1,\ldots,k.$
The space $L^2_k$ is a Hilbert space with the inner product
\begin{equation}\label{eqn4}
\langle \wt{x},\wt{y}\rangle = \frac{1}{2\pi}
\int\limits_{-\pi}^{\pi}\langle
\wt{x}(\theta),\wt{y}(\theta)\rangle\dd \theta = \frac{1}{2\pi}
\int\limits_{-\pi}^{\pi} \sum\limits_{i=1}^{k}
\overline{\wt{x}_i(\theta)}\wt{y}_i(\theta)\dd\theta.
\end{equation}
With each $\wt{A}$ in $L^\infty_{k\times k},$ we can
associate the multiplication map $M_{\wt{A}}$ on $L^2_k$
defined as
\begin{equation}\label{eqn5}
M_{\wt{A}}(\wt{x})(\theta)=\wt{A}(\theta)\wt{x}(\theta),
\end{equation}
where $\tilde{x}$ is here understood as a column vector.
It can be verified that $M_{\wt{A}}$ is a bounded 
linear operator on $L^2_k,$
and $\|M_{\wt{A}}\|\le \|\wt{A}\|.$
The space $\{M_{\wt{A}}:\wt{A}\in L^{\infty}_{k\times k}\}$ is a $C^*$-algebra, and the map $\wt{A}\mapsto M_{\wt{A}}$ is a surjective isomorphism.
This implies that
\[\|M_{\wt{A}}\|=\|\wt{A}\|\]
for all $\wt{A}$ in $L^\infty_{k\times k}.$ 

\par Next let $l^2_k$ be the set of all sequences of vectors
$\hat{x}=(x_0,x_1,x_2,\ldots),$ $x_i\in\mathbb{C}^k$ such that
$\sum\limits_{i=0}^{\infty} \|x_i\|^2 < \infty.$ 
Here $\|x_i\|$ is the Euclidean norm of
$x_i=(x_i^{(1)},\ldots,x_i^{(k)})$.  
The
space $l^2_k$ is a Hilbert space with the inner product
given by
\begin{equation}\label{eq1}
\langle \hat{x},\hat{y}\rangle = \sum\limits_{i=0}^{\infty} \langle x_i,y_i\rangle.
\end{equation}
Clearly this inner product induces the $l^2$ norm on $l^2_k.$
We denote this norm by $\|\cdot\|_2.$ 
In a similar way, $l^2_k(\mathbb{Z})$ is the Hilbert space of all square summable doubly infinite sequences $\hat{x}$ of vectors with the $l^2$ norm.

\par Throughout this paper, we denote the elements of $L^\infty_{k\times k}$($L^2_k$) by
$\wt{A},\wt{B},\ldots$ ($\wt{u},\wt{x},\ldots$), the elements of $l^2_k$ by $\hat{u},\hat{x},\ldots,$
and the usual matrices (vectors) by $A,B,\ldots$ ($u,x,\ldots$),
unless we mention otherwise.
\par Let $\wt{A}\in L^\infty_{k\times k}.$
For each $n\in\mathbb{Z},$ let $A_n$ be the $n\textrm{th}$ Fourier coefficient of $\wt{A}$  given by
$$A_n=\frac{1}{2\pi}\int\limits_{-\pi}^{\pi}\wt{A}(\theta)\e^{-\imath n\theta}\dd\theta.$$
 Suppose that $L(\wt{A})$ is the doubly infinite
$k\times k$ block Toeplitz matrix
$\begin{bmatrix}A_{i-j}\end{bmatrix}_{i,j=-\infty}^{\infty}.$ 
Since $l^2_k(\mathbb{Z})$ and $L^2_k$ are isomorphic Hilbert spaces, we
can identify $L(\wt{A})$ with the linear operator $M_{\wt{A}}$ defined
in \eqref{eqn5}.
Let  $T(\wt{A})$ be the infinite block Toeplitz matrix
$\begin{bmatrix}A_{i-j}\end{bmatrix}_{i,j=0}^{\infty}.$ This
is a principal submatrix of $L(\wt{A}).$ If for
$n\in\mathbb{N},$  $\overline{P}_n$ is the projection
operator on $l^2_k(\mathbb{Z})$ defined as
\[\overline{P}_n(\ldots,x_{-n},x_{-(n-1)},\ldots,x_0,\ldots,x_n,\ldots)
= (\ldots,0,0,x_{-(n-1)},\ldots,x_0,\ldots,x_n,\ldots),\]
then $\overline{P}_nL(\wt{A})\overline{P}_n$ converges strongly to
$L(\wt{A}),$ and for every $n,$
$\overline{P}_nL(\wt{A})\overline{P}_n=T(\wt{A}).$
For $\wt{A}$ in $L^\infty_{k\times k},$
we say $T(\wt{A})$ is the infinite block Toeplitz matrix generated by $\wt{A}$
and  $\wt{A}$ is the \emph{symbol} of the block Toeplitz operator $T(\wt{A}).$

\begin{prop}\label{prop1}
Let $\Sigma=\begin{bmatrix}A_{i-j}\end{bmatrix}_{i,j=0}^{\infty}$ be an infinite block Toeplitz matrix.
Then $\Sigma$ is a bounded linear operator on $l^2_k$ if and only if $\Sigma=T(\wt{A})$ for some $\wt{A}$ in $L^\infty_{k\times k}.$
In this case $\|\Sigma\|=\|\wt{A}\|.$
For every $\hat{x}\in l^2_k,$
\begin{equation}\label{eqn6}
\langle \hat{x},T(\wt{A})\hat{x}\rangle = \frac{1}{2\pi}
\int\limits_{-\pi}^{\pi} \langle \wt{x}(\theta),
\wt{A}(\theta)\wt{x}(\theta)\rangle \dd\theta,
\end{equation}
where $\wt{x}$ is the element of $L^2_k$ defined as
$\wt{x}(\theta)=\sum\limits_{n=0}^{\infty}x_n\e^{\imath
n\theta}.$
\end{prop}

\par A Hermitian operator $T$ on a Hilbert space $\mathcal{H}$ is said to be \emph{positive semidefinite} if $\langle x,Tx \rangle \ge 0$ for all $x$ in $\mathcal{H}$. If equality here holds only for the null vector, then $T$ is said to be \emph{positive definite}. If the space $\mathcal{H}$ is finite-dimensional, a positive semidefinite operator is positive definite if and only if it is invertible. This is not the case when $\mathcal{H}$ is infinite-dimensional (consider, e.g., the operator $T= \mathrm{diag}(1, 1/2, 1/3,.....)$ on the space $l_2$). So, we will use the term \emph{positive invertible} for an operator that is positive definite and invertible.
Let $T_n(\wt{A})$ be the truncated $n\times n$ block
Toeplitz matrix $\begin{bmatrix} A_{i-j}
\end{bmatrix}_{i,j=0}^{n-1}.$
The operator $T(\wt{A})$ is positive semidefinite if and only if all $T_n(\wt{A})$ are positive semidefinite.

\par The \emph{essential range} of $\wt{A}$ is given by the set
of all $k \times k$ matrices $B$ such
that for every $\epsilon >0,\, m(\{t: \|\wt{A}(t) -B\| <
\epsilon\})>0.$
Here $m(\cdot)$ denotes the Lebesgue measure. 
We denote the essential range of $\wt{A}$ by
$\mathcal{R}(\wt{A}).$ Clearly the essential range of $\wt{A}$ is closed in the space of $k\times k$ matrices and
is contained in the closure of the range of $\wt{A}.$
So, if $\wt{A}\in L^\infty_{k\times k},$ then
$\mathcal{R}(\wt{A})$ is compact. Also if $X\subseteq
[-\pi,\pi]$ is any set such that $\wt{A}(X)\cap
\mathcal{R}(\wt{A})=\emptyset,$ then $m(X)=0.$

\begin{prop}\label{prop3}
Let $\wt{A}\in L^\infty_{k\times k}.$
Then $T(\wt{A})$ is a positive semidefinite operator on $l^2_k$
if and only if all matrices $\wt{A}(\theta)$ in $\mathcal{R}(\wt{A})$ are positive semidefinite.
Consequently the matrices $T_n(\wt{A})$ are positive semidefinite for all $n$ if and only if all matrices $\wt{A}(\theta)$ in $\mathcal{R}(\wt{A})$ are positive semidefinite.

If $T(\wt{A})$ is positive invertible,
then all matrices $\wt{A}(\theta)$ in $\mathcal{R}(\wt{A})$ are positive definite.
\end{prop}

%

\par The following proposition gives an equivalent condition for $T_n(\wt{A})$ to be positive definite for each $n.$
See \cite{Miranda-Tilli}.

\begin{prop}\label{propf5}
For every $\wt{A}\in L^\infty_{k\times k},$ $T_n(\wt{A})$ is positive definite for every $n$ if and only if
all matrices $\wt{A}(\theta)$ in $\mathcal{R}(\wt{A})$ are positive semidefinite, and
$\wt{A}(\theta)$ are positive definite for all $\theta$ in some subset of $[-\pi,\pi]$ that has positive measure
\end{prop}

\par We call a Toeplitz operator $\Sigma=\begin{bmatrix}A_{i-j}\end{bmatrix}$  {\it partially symmetric} if each
$A_n$ is a real matrix and $A_{-n}=A_n$ for all $n\in\mathbb{N}.$
An element $\wt{A}$ of $L^\infty_{k\times k}$ is \emph{even} if $\wt{A}(-\theta)=\wt{A}(\theta)$ for almost all $\theta\in [-\pi,\pi].$

\begin{prop}\label{propps}
For any $\wt{A}$ in $L^\infty_{k\times k},$ the following statements are equivalent.
\begin{itemize}
\item[(i)] $T(\wt{A})$ is partially symmetric.
\item[(ii)] $\wt{A}$ is even and every matrix $\wt{A}(\theta)$ in $\mathcal{R}(\wt{A})$ is real.
\item[(iii)] The infinite matrix $T(\wt{A})$ is real and every matrix $\wt{A}(\theta)$ in $\mathcal{R}(\wt{A})$ is real.
\end{itemize}
\end{prop}

\par Here we point out that symplectic eigenvalues of $T_n(\wt{A})$ and $\wt{A}(\theta)$
 are defined only when $T(\wt{A})$ is a partially symmetric operator on $l_{2k}^2$, and $T_n(\wt{A})$ and $\wt{A}(\theta)$ are positive definite.

\par A stationary G-chain $\Sigma=\begin{bmatrix}A_{i-j}\end{bmatrix}$ 
is \emph{bounded} if it is 
bounded as a linear operator on $l^2_{2k},$
and is \emph{partially symmetric} if it is a partially symmetric linear operator.
The following theorem gives a characterisation of a 
partially symmetric bounded stationary G-chain
in terms of its symbol.

\begin{thm}\label{thm4}
Let $\Sigma$ be an infinite real matrix. 
Then $\Sigma$ is a partially symmetric bounded stationary G-chain if
and only if it is generated by an $\wt{A}$ in $L^\infty_{2k\times 2k}$ such that $\wt{A}(\theta)$ is a G-matrix for all
$\wt{A}(\theta)$ in $\mathcal{R}(\wt{A}).$
\end{thm}

\begin{proof}
By Proposition \ref{prop1} $\Sigma$ is a bounded linear operator on $l^2_{2k}$ if and only if $\Sigma=T(\wt{A})$ for some $\wt{A}$ in $L^\infty_{2k\times 2k}.$
By Propositions \ref{prop3} and \ref{propps} we know that $\wt{A}(\theta)$ are real positive semidefinite
matrices for all $\wt{A}(\theta)\in\mathcal{R}(\wt{A})$ if and only if
$T(\wt{A})$ is partially symmetric and positive semidefinite.
Let $\Sigma_0=T(\wt{A})+\frac{\imath}{2}J_\infty,$
where $J_\infty$ is the infinite block diagonal matrix $\oplus_{\mathbb{N}}J_2.$
Clearly $\Sigma_0$ is the infinite $2k\times 2k$ block Toeplitz matrix
corresponding to the sequence $\langle
B_n\rangle_{n\in\mathbb{Z}},$ where 
\[B_n=\begin{cases}
A_n & n\ne 0,\\
A_0+\frac{\imath}{2}J_{2k} & n=0.
\end{cases}\]
One can see that $\Sigma_0$ is generated by the function
$\wt{B}=\wt{A}+\frac{\imath}{2}J_{2k}.$ Now $\wt{A}(\theta)$ is a G-matrix
if and only if $\wt{B}(\theta)$ is positive semidefinite.
Similarly, $T(\wt{A})$ is a G-chain if and only if $T_n(\wt{B})$
is positive semidefinite for every $n\in\mathbb{N}.$ Hence we obtain the theorem by
using Proposition \ref{prop3}.
\end{proof}


\section{Symplectic numerical range} \label{S3}

\par Let $\mathcal{H}$ be a real separable Hilbert space. We denote
the direct sum $\mathcal{H} \oplus \mathcal{H}$ by $\wt{\mathcal{H}}.$ It is easy to
see that the space $\wt{\mathcal{H}}$ is isomorphic to
$\bigoplus_\mathbb{N} \mathcal{K}$ where $\mathcal{K}$ is a two dimensional real
Hilbert space and the operator $J = \begin{bmatrix} 0
& I \\ -I & 0 \end{bmatrix}$ on $\wt{\mathcal{H}}$ is orthogonally
equivalent to $\bigoplus_{\mathbb{N}} J_2.$
Henceforth we will identify $\wt{\mathcal{H}}$  with  $\bigoplus_\mathbb{N} \mathcal{K}$
and the operator $J$ with $\bigoplus_{\mathbb{N}} J_2.$

\begin{defin}
Let $A$ be a positive definite operator on
$\wt{\mathcal{H}}.$ We define the symplectic numerical range
of $A$ to be the set 
\[W_s(A) = \left\{ \frac{1}{2} ( \langle u, A u\rangle +
\langle v, Av \rangle): \langle u,Jv\rangle=1,\, u, \, v \in \wt{\mathcal{H}}\right\}.\]
\end{defin}

This is a subset of $(0,\infty)$. It is unbounded as the set of vectors $(u,v)$  with $\langle u, J v \rangle =1$ is unbounded.
An infinite dimensional version of Williamson's theorem 
was proved in \cite{tiju}:
for any
positive invertible operator $A$ on $\wt{\mathcal{H}}$ there exists a positive invertible
operator $P$ on  $\mathcal{H}$
and a symplectic transformation $L: \wt{\mathcal{H}} \to \wt{\mathcal{H}}$ such that 
\[ A = L \begin{bmatrix} P & 0 \\ 0 & P \end{bmatrix} L^T.\]
The \emph{symplectic spectrum} of $A$ is the spectrum of the positive
invertible operator  $P.$
If $A$ is a $2n\times 2n$ real positive definite matrix, then
its symplectic spectrum is the set of its symplectic eigenvalues
$\{d_1(A),\ldots,d_n(A)\}\subseteq (0,\infty).$
We denote by
$\sigma_s(A)$ the symplectic spectrum of $A.$

\begin{prop} \label{prop4}
Let $\mathcal{H}$ be a real separable Hilbert space and $A$ a
bounded positive invertible operator on $\wt{\mathcal{H}}.$ Then
\begin{enumerate}[(i)]
\item $W_s(A) = W_s(MAM^T)$ for every symplectic transformation $M.$
\item $\sigma_s(A)\subseteq\overline{W_s(A)}=[\inf \sigma_s(A),\infty).$
\item If $\mathcal{H}$ is finite-dimensional,
then $W_s(A)$ is the closed set $[d_1(A),\infty),$
where $d_1(A)$ is the minimum symplectic eigenvalue of $A.$
\end{enumerate}
\end{prop}

\begin{proof}
Part (i)  follows from the fact that
$\langle Mu,JMv\rangle=\langle u,Jv\rangle$
for every symplectic transformation $M.$

\par Let $u,v\in\wt{\mathcal{H}}$ be such that $\langle u,Jv\rangle =1.$
Let $\alpha=\frac{1}{2}\left(\langle u,Au\rangle+\langle v,Av\rangle\right),$ $\alpha_1=\langle u,Au\rangle/2$ and $\alpha_2=\langle v,Av\rangle/2.$
For any $t>0,$ $\langle tu,Jv/t\rangle=1.$
Let
\begin{equation*}
\alpha(t)=\frac{\langle tu,A(tu)\rangle+\langle v/t,A(v/t)\rangle}{2}
= t^2\alpha_1+\frac{1}{t^2}\alpha_2.
\end{equation*}
Clearly $\alpha(t)$ is continuous in $t$ and $\alpha(1)=\alpha.$
Since $\lim_{t\to\infty}\alpha(t)=\infty,$
by the intermediate value theorem, $[\alpha,\infty)\subseteq W_s(A).$
Thus $\overline{W_s(A)}=[\inf W_s(A),\infty).$
Now, let $P$ be the positive invertible operator on $\mathcal{H}$ such that
$$A=L\hat{P}L^T,$$
where $L$ is a symplectic transformation on $\wt{\mathcal{H}}$ and $\hat{P}=\begin{bmatrix}P & O\\O & P\end{bmatrix}.$
We know that $\sigma_s(A)=\sigma(P)=\sigma(\hat{P}).$
Thus, we only need to show that
$\inf W_s(A)=\inf \sigma(\hat{P}).$
Let $u,v$ be two distinct unit vectors in $\wt{\mathcal{H}}.$
Without loss of generality, we can assume that
$\langle u,Jv\rangle>0.$
Clearly $\langle u,Jv\rangle\le 1.$
Let $u_0=u/\sqrt{\langle u,Jv\rangle}$ and $v_0=v/\sqrt{\langle u,Jv\rangle}.$
Then $\langle u_0,Jv_0\rangle=1,$ and
\begin{equation}
\frac{\langle u_0,Au_0\rangle+\langle v_0,Av_0\rangle}{2}=\frac{\langle u,Au\rangle+\langle v,Av\rangle}{2\langle u,Jv\rangle} 
\ge\frac{\langle u,Au\rangle+\langle v,Av\rangle}{2}.\label{eqntt1}
\end{equation}
Since the left-hand side of \eqref{eqntt1} belongs to $W_s(\hat{P})$
and the right-hand side to $W(\hat{P}),$ it follows that
\begin{equation}
\inf W_s(\hat{P})\ge\inf W(\hat{P}).\label{eqnt2}
\end{equation}
Now let $x$ be any unit vector in $\mathcal{H},$
and let $u=\frac{1}{\sqrt{2}}(x\oplus x)$ and $v=\frac{1}{\sqrt{2}}(-x\oplus x)$. 
Then $\langle u,Jv\rangle=1.$
We see that
$$\langle x,Px\rangle=\frac{\langle u,\hat{P}u\rangle+\langle v,\hat{P}v\rangle}{2}.$$
This implies that
\begin{equation}
\inf W(P)\ge \inf W_s(\hat{P}).\label{eqnt3}
\end{equation}
Combining \eqref{eqnt2} and \eqref{eqnt3}, and using the fact that
$\inf W(\hat{P})=\inf W(P)=\inf \sigma(P),$
we obtain (ii).

\par When $\mathcal{H}$ is finite-dimensional,  we have 
$$d_1(A)={\underset{\underset{\langle u,Jv\rangle=1}{u,v\in\mathcal{H}}}{\min}}\frac{\langle u,Au\rangle+\langle v,Av\rangle}{2}.$$
(See Theorem 5 of \cite{bhatia-jain}.)
This gives part (iii).
\end{proof}

\par Let $\wt{A}\in L^\infty_{2k\times 2k}$ be such that all matrices $\wt{A}(\theta)$ in $\mathcal{R}(\wt{A})$ are real positive definite. Then the symplectic numerical range of
$\wt{A}$ is the set 
\begin{multline*}
W_s(\wt{A}) = \left\{ \frac{1}{2\pi} \int_{-\pi}^\pi
\frac{\langle \wt{u}(\theta), \wt{A}(\theta)\wt{u}(\theta)\rangle +
\langle \wt{v}(\theta), \wt{A}(\theta)\wt{v}(\theta)\rangle}{2} \dd \theta
: \right.\\
\left. \frac{1}{2\pi}
\int_{-\pi}^\pi \langle \wt{u}(\theta),
J\wt{v}(\theta)\rangle \dd \theta
=1,\;\wt{u}, \wt{v} \in L_{2k}^2\right\}.
\end{multline*}

\par We next give a relationship between
$W_s(\wt{A})$ and $W_s(\wt{A}(\theta))$ for $\wt{A}(\theta)\in\mathcal{R}(\wt{A}).$ 

\begin{thm}\label{the1}
Let $\wt{A}$ be an element of $L^\infty_{2k\times 2k}$ such that all matrices $\wt{A}(\theta)$ in $\mathcal{R}(\wt{A})$ are real positive definite.
The set $\overline{W_s(\wt{A})}$ is the same as the closed convex hull of $\bigcup_{\wt{A}(\theta)\in \mathcal{R}(\wt{A})}
W_s(\wt{A}(\theta)).$
\end{thm}

\begin{proof}
Let $B=\wt{A}(\theta)\in \mathcal{R}(\wt{A})$ and $\mu \in W_s(B).$ Then there
exists a pair $(u,v)$ in $\mathbb{R}^{2k}\times\mathbb{R}^{2k}$ such that $\langle u, J_{2k}v\rangle =1$ and $\mu = \frac{1}{2} (\langle u, Bu\rangle +
\langle v, Bv\rangle).$ 
For $n \in \mathbb{N},$  let $S_n$ be the set

\[S_n= \left\{ t: \|\wt{A}(t) -B\| <\frac{1}{n(\|u\|^2
+\|v\|^2)}\right\},\]
and let $m_n =m(S_n)$, the measure of $S_n$. 
Since $B\in\mathcal{R}(\wt{A}),$ $m_n>0$ for every $n.$
Define the vector functions $\wt{u}_n$ and $\wt{v}_n$ on $[-\pi,\pi]$ as
\[ \wt{u}_n(t) = \begin{cases}
\sqrt{\frac{2\pi}{m_n}}u & t\in S_n\\
0 & \text{otherwise,}
\end{cases}
\text{\qquad and \qquad}
\wt{v}_n(t) = \begin{cases}
\sqrt{\frac{2\pi}{m_n}}v & t\in S_n\\
0 & \text{otherwise.}
\end{cases}\]
Clearly $\wt{u}_n$ and $\wt{v}_n$ are in $L_{2k}^2,$
and $\langle\wt{u}_n,J_{2k}\wt{v}_n\rangle=1.$
Let
$\mu_n = \frac{1}{2} ( \langle \wt{u}_n, \wt{A} \wt{u}_n
\rangle + \langle \wt{v}_n, \wt{A} \wt{v}_n \rangle).$ 
Then $\mu_n \in W_s(\wt{A}),$  and we have
\begin{align*}
|\mu_n - \mu| &=
\left|\frac{1}{2\pi}\int_{-\pi}^{\pi}\frac{1}{2}\left(\langle
\wt{u}_n(t),\wt{A}(t)\wt{u}_n(t)\rangle+\langle\wt{v}_n(t),\wt{A}(t)\wt{v}_n(t)\rangle\right)
\dd t  -  \frac{1}{2}\left(\langle u,Bu\rangle+\langle v,Bv\rangle\right)\right|\\
&= \left|\frac{1}{m_n}\int_{S_n}\frac{1}{2}\left(u,\wt{A}(t)u\rangle+\langle v,\wt{A}(t)v\rangle\right)\dd t  - \frac{1}{m_n}\int_{S_n}\frac{1}{2}\left(\langle u,Bu\rangle+\langle v,Bv\rangle\right)\dd t\right|\\
&= \left| \frac{1}{m_n} \int_{S_n} \frac{1}{2}
\left(\langle u, (\wt{A}(t)-B) u \rangle + \langle v, (\wt{A}(t)-B) v \rangle \right) \dd
t \right|\\
&\le \frac{1}{2m_n} \left(\|u\|^2 + \|v\|^2 \right) \int_{S_n} \left\|\wt{A}(t) -
B \right\| \dd t\\
&\le \frac{1}{2n}. 
\end{align*}
This proves $\mu_n\to \mu.$ Hence $W_s(B) \subseteq \overline{W_s(\wt{A})}.$
 Since $W_s(\wt{A})$
is convex, the closed convex hull of
$\bigcup_{B \in \mathcal{R}(\wt{A})} W_s(B)$ is contained in $\overline{W_s(\wt{A})}.$

\par To prove the reverse inclusion, we use the fact that every element of $W_s(\wt{A})$ is a
limit of  finite sums of the form 
\[ \sum_j \frac{\alpha_j}{4\pi} ( \langle u_j, \wt{A}(\theta_j) u_j
\rangle + \langle v_j, \wt{A}(\theta_j) v_j \rangle ),\]
where $\wt{A}(\theta_j) \in \mathcal{R}(\wt{A})$, and $\alpha_j \ge 0$ 
are such that $\sum_j \alpha_j \langle u_j, J v_j \rangle =1.$
Let $\beta_j = \langle u_j, J v_j \rangle .$ Without loss of
generality we may assume that $\beta_j \ge 0$ for all $j.$
Replacing $u_j$ by $\sqrt{\beta_j} u_j$ and $v_j$ by
$\sqrt{\beta_j} v_j$ we can take $\langle u_j, J_{2k} v_j \rangle
=1$ for every $j,$ and $\sum \alpha_j =1.$
This shows that every element of $W_s(\wt{A})$ is a limit of
convex combinations of elements of $\bigcup_{\wt{A}(\theta)\in\mathcal{R}(\wt{A})} W_s(\wt{A}(\theta)).$
\end{proof}

\par Let $\hat{u}, \, \hat{v} \in l_{2k}^2,$  $\hat{u}=(u_1, u_2, \ldots)$
and $\hat{v}=(v_1, v_2, \ldots).$
Define $\wt{u}(\theta)=
\sum u_n e^{\imath n \theta}$ and $\wt{v}(\theta)=
\sum v_n e^{\imath n \theta}.$ Clearly 
\[\langle \hat{u}, J \hat{v} \rangle = \frac{1}{2\pi} \int_{-\pi}^\pi
\langle \wt{u}(\theta), J_{2k} \wt{v}(\theta) \rangle \dd
\theta. \]
Using \eqref{eqn6} we see that 
\[W_s\left(T(\wt{A})\right) \subseteq W_s(\wt{A})\]
for every partially symmetric, bounded, positive invertible operator $T(\wt{A})$ on $l^2_{2k}.$
Since
$T_n(\wt{A})$ is a principal submatrix of $T(\wt{A}),$ we have 
\begin{equation}
W_s\left(T_n(\wt{A})\right) \subseteq W_s\left(T_{n+1}(\wt{A}) \right)\subseteq W_s\left(T(\wt{A})\right) \subseteq W_s(\wt{A}).
\end{equation}

Let $\wt{A}\in L^\infty_{2k\times 2k}$ be such that all matrices $\wt{A}(\theta)$ in $\mathcal{R}(\wt{A})$ are real positive definite.
Let 
\begin{align}
\mathfrak{m}_{\wt{A}} &={\underset{\theta \in[-\pi,\pi]}{\textrm{essinf}}}
d_1(\wt{A}(\theta)).\label{eqnt1}
\end{align}
Using Theorem \ref{the1} we
see that 
\begin{equation}
\mathfrak{m}_{\wt{A}} = \inf W_s(\wt{A}).\label{eqt1}
\end{equation}


\begin{thm}\label{the2}
Let $T(\wt{A})$ be a partially symmetric, bounded, positive invertible operator on $l^2_{2k}.$
Let $n\in\mathbb{N},$ and let $d$ be a symplectic eigenvalue of $T_n(\wt{A}).$
Then $d\ge \mathfrak{m}_{\wt A}.$
If $d=\mathfrak{m}_{\wt{A}},$ then $d_1(\wt{A}(\theta))$ is the constant $\mathfrak{m}_{\wt{A}}$ for almost all $\theta.$
\end{thm}

\begin{proof}
Let $d$ be a symplectic eigenvalue of $T_n(\wt{A}).$ Then there
exist vectors $\hat{u} = (u_1, \cdots , u_n), \, \hat{v}= (v_1,
\cdots, v_n),$ $u_j, \, v_j \in \mathbb{R}^{2k}$
with  $\langle \hat{u}, J \hat{v} \rangle =1$ and 
\[d= \frac{1}{2} (\langle \hat{u}, T_n(\wt{A}) \hat{u} \rangle + \langle \hat{v},
T_n(\wt{A}) \hat{v} \rangle).\]
Let $\wt{u}$ and $\wt{v}$ be the elements of $L_{2k}^2$ given by
$\wt{u}(\theta) = \sum u_j e^{\imath j \theta}, \,
\wt{v}(\theta) = \sum v_j e^{\imath j \theta}.$ 
Then 
\begin{align*}
d - \mathfrak{m}_{\wt{A}} &= \frac{1}{2} (\langle \hat{u}, T_n(\wt{A}) \hat{u}
\rangle + \langle \hat{v}, T_n(\wt{A}) \hat{v} \rangle) -
\mathfrak{m}_{\wt{A}}\\
& =  \frac{1}{2\pi} \int_{-\pi}^\pi \left[ \frac{1}{2} \left(\left\langle
\wt{u}(\theta), \wt{A}(\theta) \wt{u}(\theta) \right\rangle + \left\langle \wt{v}(\theta),
\wt{A}(\theta) \wt{v}(\theta) \right\rangle\right)  -
\mathfrak{m}_{\wt{A}}\left\langle \wt{u}(\theta),J_{2k}\wt{v}(\theta)\right\rangle  \right) \dd
\theta. 
\end{align*}
Using \eqref{eqt1}, we know that the above integrand is nonnegative almost everywhere.
Hence, we have $d \ge \mathfrak{m}_{\wt{A}}.$
Also $d=\mathfrak{m}_{\wt{A}}$ if and only if 
\[\frac{1}{2}
\left(\langle
\wt{u}(\theta),\wt{A}(\theta)\wt{u}(\theta)\rangle+\langle\wt{v}(\theta),\wt{A}(\theta)\wt{v}(\theta)\rangle
\right)=\mathfrak{m}_{\wt{A}}\langle \wt{u}(\theta),J_{2k}\wt{v}(\theta)\rangle\] 
for almost all $\theta.$
So the last statement of the theorem by using Proposition \ref{prop4}(iii) and \eqref{eqnt1}.
\end{proof}

\section{A Szeg\H o type theorem for symplectic eigenvalues
and applications} \label{S4}

\par We first recall some basic facts about symplectic
eigenvalues. See \cite{bhatia-jain, gosson}, and \cite{jain-h} for 		details.
A positive number $d$ is a symplectic eigenvalue of a
$2k\times 2k$ positive definite matrix $A$ if and only if
$\pm d$ are the eigenvalues of the (non-Hermitian) matrix $\imath
J_{2k}A.$ 
Thus each symplectic eigenvalue
$d_i$ of $A$ lies in the interval $[0,\|A\|].$
Let $T(\wt{A})$ be a bounded, partially symmetric, positive invertible operator on $l^2_{2k}$ generated by $\wt{A}.$
Let $d_1^{(n)}\le\cdots\le d_{nk}^{(n)}$ denote the symplectic eigenvalues of $T_n(\wt{A})$ arranged in increasing order.
 Since $\wt{A}\in L^\infty_{2k\times 2k},$
each $d_i^{(n)}\le \|T_n(\wt{A})\|\le\|\wt{A}\|.$

\begin{thm}\label{thm_main}
Let $\Sigma$ be a partially symmetric, bounded positive invertible operator on $l^2_{2k}$
generated by $\wt{A}.$
Let $d_1^{(n)} \le \cdots \le d_{nk}^{(n)}$
denote the symplectic eigenvalues of $T_n(\wt{A}).$ Then for
every function $f$ continuous on $[0,\|\wt{A}\|],$
\begin{equation}\label{eqn_main}
\lim\limits_{n\to\infty} \frac{1}{n} \sum_{j=1}^{nk} f(d_j^{(n)}) =
\frac{1}{2\pi} \int_{-\pi}^{\pi} \sum_{j=1}^{k}
f\left(d_j(\wt{A}(\theta))\right) \dd\theta.
\end{equation}
\end{thm}

\begin{proof}
By Theorem 6.24 of \cite{bs}, we know that if $\wt{B}\in
L^\infty_{2k\times 2k},$ and $\lambda_1^{(n)}, \ldots,
\lambda_{2nk}^{(n)}$ are the eigenvalues of the $n\times n$
truncated block Toeplitz matrix $T_n(\wt{B}),$ then
\begin{equation}\label{eqna1}
\lim\limits_{n\to\infty} \frac{1}{n} \sum\limits_{j=1}^{2nk}
\left(\lambda_j^{(n)}\right)^{m} = \frac{1}{2\pi}
\int_{-\pi}^{\pi} \sum\limits_{j=1}^{2k}
\left(\lambda_j(\wt{B}(\theta))\right)^{m} \dd\theta,
\end{equation}
for every nonnegative integer $m.$
Suppose $\wt{B}=\imath J_{2k}\wt{A}.$
Then $T_n(\wt{B})=\imath J_{2nk}T_n(\wt{A}),$ and the
eigenvalues $\lambda_1^{(n)}, \ldots, \lambda_{2nk}^{(n)}$
of $T_n(\wt{B})$ are $\pm d_1^{(n)},\ldots,\pm
d_{nk}^{(n)}.$
Also the eigenvalues $\lambda_j(\wt{B}(\theta))$ of $\wt{B}(\theta)$ are $\pm d_j(\wt{A}(\theta)).$
Hence for every non-negative integer $m,$
\begin{equation}\label{eqna2}
\lim_{n\to\infty} \frac{1}{n} \sum_{j=1}^{nk}
\left(d_j^{(n)}\right)^{2m} = \frac{1}{2\pi}
\int_{-\pi}^{\pi} \sum_{j=1}^{k}
\left(d_j(\wt{A}(\theta))\right)^{2m} \dd\theta.
\end{equation}
By linearity, we can extend \eqref{eqna2} to polynomials in
$\left(d_j^{(n)}\right)^2,$ i.e.,
\begin{equation}\label{eqna3}
\lim_{n\to\infty} \frac{1}{n} \sum_{j=1}^{nk} p
\left({d_j^{(n)}}^2\right) = \frac{1}{2\pi}
\int_{-\pi}^{\pi} \sum_{j=1}^{k} p \left(d_j^2(\wt{A}(\theta))\right) \dd\theta
\end{equation}
for every polynomial $p.$ Each $d_j^{(n)}\in
\left[0,\|\wt{A}\|\right],$  and hence $\left(d_j^{(n)}\right)^2\in
\left[0,\|\wt{A}\|^2\right].$
Let $q(x)$ be any polynomial and let $s(x)=q(\sqrt{x}).$
Clearly $s$  is continuous on $\left[0,\|\wt{A}\|^2\right].$
For a given $\epsilon>0,$ we can find a polynomial $p$ such that
\[\sup_{x\in \left[0,\|\wt{A}\|^2\right]} |s(x)-p(x)|
<\epsilon.\]
Since $d_j(\wt{A}(\theta)) \le \|\wt{A}(\theta)\|$ and
$\|\wt{A}(\theta)\|
\le \|\wt{A}\|$ for almost all $\theta,$
\begin{multline}\label{eqna4}
\left\vert \frac{1}{2\pi} \int_{-\pi}^{\pi}
s\left(d_j^2(\wt{A}(\theta))\right) \dd\theta - \frac{1}{2\pi}
\int_{-\pi}^{\pi} p\left(d_j^2(\wt{A}(\theta))\right) \dd\theta
\right\vert \\ \le \frac{1}{2\pi} \int_{-\pi}^{\pi}
\left\|s\left(d_j^2(\wt{A}(\theta))\right) -
p\left(d_j^2(\wt{A}(\theta))\right) \right\| \dd\theta \; \le \epsilon,
\end{multline}
holds for all $j=1,\ldots,k.$
Similarly,\begin{equation}\label{eqna5}
\left|s\left({d_j^{(n)}}^2\right) -
p\left({d_j^{(n)}}^2\right)\right| < \epsilon
\end{equation}
for all $j=1,\ldots,nk.$
Combining the relations \eqref{eqna3}, \eqref{eqna4}, and
\eqref{eqna5}, we see that
\[\lim_{n\to\infty} \frac{1}{n} \sum_{j=1}^{nk}
s\left({d_j^{(n)}}^2\right) = \frac{1}{2\pi}
\int_{-\pi}^{\pi} s\left(d_j^2(\wt{A}(\theta))\right) \dd\theta.\]
Since $s(x^2)=q(x),$ we have
\begin{equation}\label{eqna6}
\lim_{n\to\infty} \frac{1}{n} \sum_{j=1}^{nk}
q\left(d_j^{(n)}\right) = \frac{1}{2\pi} \int_{-\pi}^{\pi}
q\left(d_j(\wt{A}(\theta))\right) \dd\theta.
\end{equation}
Now let $f$ be any continuous function on $\left[0,
\|\wt{A}\|\right].$ Then by using the Weierstrass
approximation theorem and arguing as above, we can show that
\eqref{eqn_main} holds for $f.$
\end{proof}

\begin{rem}
\emph {
We know that $d$ is a symplectic eigenvalue of a positive definite matrix $A$ if and only if $\pm d$ are eigenvalues of the non-Hermitian matrix $iJA$. In \cite{tilli-2} Tilli has proved a very general version of Szeg\H o's limit theorem for non-Hermitian block Toeplitz matrices. It is possible to derive Theorem \ref{thm_main} from Tilli's general results. The proofs of the general version are, naturally, more intricate. We have given a short self-contained presentation for the special case we need.
}
\end{rem}

\par The entropy rate of a stationary quantum Gaussian process with the associated G-chain $T(\wt{A})$ is given by the formula
\[\mathfrak{S}(T(\wt{A}))=\lim_{n \to \infty} \frac{S(T_n(\wt{A}))}{n},\]
where $S(T_n(\wt{A}))$ denotes the entropy of the quantum Gaussian state with the corresponding G-matrix $T_n(\wt{A}).$
As a consequence of Theorem \ref{thm_main}, we obtain a closed expression for the
entropy rate of a partially symmetric, bounded stationary quantum Gaussian process in terms of the entropies
of the Gaussian states with G-matrices $\wt{A}(\theta).$

\begin{cor}\label{cor4.2}
Let $T(\wt{A})$ be a partially symmetric, bounded stationary G-chain generated by $\wt{A}.$
The entropy rate $\mathfrak{S}\left(T(\wt{A})\right)$ of the corresponding stationary Gaussian process is
\begin{equation}\label{eq:h}
\mathfrak{S}\left(T(\wt{A})\right) = \frac{1}{2\pi} \int_{-\pi}^{\pi}
S\left(\wt{A}(\theta)\right) \dd\theta.
\end{equation}
\end{cor}

\begin{proof}
Define the function $f:[0,\|\wt{A}\|]\to\mathbb{R}$ as
$$f(x)=\begin{cases}
\left(x+\frac{1}{2}\right)\log\left(x+\frac{1}{2}\right)-\left(x-\frac{1}{2}\right)\log\left(x-\frac{1}{2}\right) & \quad \text{ if }\frac{1}{2}<x\le\|\wt{A}\|,\\
0 & \quad \text{ if } 0\le x\le \frac{1}{2}.
\end{cases}$$
Using \eqref{eqn1}we can see that the entropy of any G-matrix $B$ can be written as
\begin{equation}
S(B)=\sum\limits_{d\in\sigma_s(B)}f(d).\label{eqer1}
\end{equation}
Hence the entropy rate $\mathfrak{S}\left(T(\wt{A})\right)$ of $T(\wt{A})$ is given by
\begin{equation}
\mathfrak{S}\left(T(\wt{A})\right)=\lim\limits_{n\to\infty}\frac{1}{n}\sum\limits_{j=1}^{nk}f\left(d_j(T_n(\wt{A}))\right).\label{eqer2}
\end{equation}
Since $f$ is continuous and $\wt{A}\in L^\infty_{2k\times 2k},$ we can apply Theorem \ref{thm_main} to get
$$\mathfrak{S}\left(T(\wt{A})\right)=\frac{1}{2\pi}\int\limits_{-\pi}^{\pi}\frac{1}{k}\sum\limits_{j=1}^{k}f\left(d_j(\wt{A}(\theta))\right)\dd\theta.$$
Using the formula \eqref{eqer1} for the sum inside this integral, we obtain \eqref{eq:h}.
\end{proof}

\par The entropy rate of a special kind of stationary quantum Gaussian process has been computed in the paper \cite{krprb2}. There the authors considered a block Toeplitz matrix $T_n(\wt{A})$ given by 
\[T_n(\wt{A}) = \begin{cases}
A & n=0,\\
p_{|n|} B & \text{otherwise};
\end{cases}\]
where $A$ and  $B$ are $2k \times 2k$ real symmetric matrices such that $A + tB$ is a G-matrix for each $t \in [-2,2]$, and $\{p_1, p_2, \ldots\} $ is a probability distribution over $\{1,2, \ldots\}$. In this case $\wt{A}(\theta)$ takes the form $\left( A + \sum_{j \in \mathbb{Z}\backslash\{0\}} p_{|j|} B e^{\imath j \theta}\right),\, \theta \in [-\pi, \pi].$  Our Corollary \ref{cor4.2} gives a much more general result. 


\par In the rest of the paper, we use Theorem \ref{thm_main} to study the distribution of symplectic eigenvalues of truncated block Toeplitz matrices. 
Henceforth $T(\wt{A})$ is a partially symmetric, bounded, positive invertible operator on $l^2_{2k}$ generated by $\wt{A}.$ Recall the definition of $\mathfrak{m}_{\wt{A}}$ given in \eqref{eqnt1}.

\begin{thm}\label{the22}
For each $m\in\mathbb{N},$ $\lim_{n \to \infty} d_m^{(n)} = \mathfrak{m}_{\wt{A}}.$ 
Consequently 
\begin{equation}\label{s4}
\overline{\bigcup_{n \in \mathbb{N}} W_s(T_n(\wt{A})) } = \overline{W_s(\wt{A})}.
\end{equation}
\end{thm}
\begin{proof}
For every $n\in\mathbb{N},$ $T_n(\wt{A})$ is a principal submatrix of $T_{n+1}(\wt{A}).$
Hence from the relation (42) of \cite{bhatia-jain} we see that
\[0\le d_m^{(n+1)} \le d_m^{(n)}\ \ \textrm{ for all }n\ge m.\]
 Hence $\lim\limits_{n\to\infty}d_m^{(n)}$ exists. Suppose this equals
$ r .$ By definition $r \ge \mathfrak{m}_{\wt{A}}$.  Suppose  $ r  >
\mathfrak{m}_{\wt{A}}.$ Define $f$ on $[0, \|\wt{A}\|]$ as 
\[f(x) = \begin{cases} 
0 & x>  r  \text{ or } x< \mathfrak{m}_{\wt{A}} -1,\\
x - \mathfrak{m}_{\wt{A}} +1 & \mathfrak{m}_{\wt{A}} -1 \le x \le
\mathfrak{m}_{\wt{A}}, \\
\frac{x-  r }{\mathfrak{m}_{\wt{A}} -  r }  &
\mathfrak{m}_{\wt{A}} \le x \le  r .
\end{cases} \]
Clearly $f$ is continuous on $[0, \|\wt{A}\|],$ and the formula 
\eqref{eqn_main} holds for $f.$
But $f(d_j^{(n)}) =0$ for all
$j=m, \cdots, nk$ and for all $n \ge m.$ 
So the left hand side of \eqref{eqn_main} is zero.
But since $\mathfrak{m}_{\wt{A}} = \textrm{essinf}\, d_1(\wt{A}(\theta))$
and $f$ is positive for $[\mathfrak{m}_{\wt{A}},  r ],$ we
have $f\left( d_j(\wt{A}(\theta))\right) > 0$ for $\theta$ in a
set of positive measure. Hence the right hand side of \eqref{eqn_main} is strictly positive. This is a contradiction.  

\par By Proposition \ref{prop4} (iii), we know that
$W_s(T_n(\wt{A})) =[d_1^{(n)}, \infty).$ Hence 
\[ \overline{\bigcup_{n \in \mathbb{N}} W_s(T_n(\wt{A}))} = [\lim_{n \to \infty}
d_1^{(n)}, \infty) = [\mathfrak{m}_{\wt{A}}, \infty).\]
Since $\mathfrak{m}_{\wt{A}} = \inf W_s(\wt{A})$ and $W_s(\wt{A})$ is convex,
this proves \eqref{s4}
\end{proof}

\begin{lemma}\label{ls5}
Let $K$ be a compact subset of $\mathbb{R},$ and let
$c_n(K)$ be the cardinality of the set $\{j: d_j^{(n)} \in
K\}.$ Then 
\begin{equation}\label{s5}
\lim_{n \to \infty} \frac{c_n(K)}{n} = \frac{1}{2\pi}\sum\limits_{j=1}^k
m\{ \theta \in [-\pi, \pi]: d_j\left(\wt{A}(\theta)\right) \in
K\}.
\end{equation}
\end{lemma}
\begin{proof}
Let $g$ be the distance function defined as $$g(x) = |x -K| := \inf \{|x -t|: t \in K\}.$$
For any $\epsilon>0$ define the function $f_\epsilon : [0 , \|\wt{A}\|] \to  \mathbb{R}$ as
$f_\epsilon(x) = \e^{-\frac{g(x)}{\epsilon}}.$
Clearly $f_\epsilon$ is continuous and $f_\epsilon(x)=1$ if
and only if $x \in K.$ We can see that as $\epsilon \to 0,$
$f_\epsilon$ converges to the characteristic function
$\chi_K$ in the $L^1$ norm. Hence 
\begin{align*}
\lim_{\epsilon \to 0} \frac{1}{2\pi} \int_{-\pi}^\pi
\sum_{j=1}^k  f_\epsilon (d_j(\wt{A}(\theta))) \dd \theta  &=
\frac{1}{2\pi} \int_{-\pi}^\pi \sum_j \chi_K
(d_j(\wt{A}(\theta))) \dd \theta \\
&= \frac{1}{2\pi} \sum_j m(\{\theta:d_j(\wt{A}(\theta)) \in
K\}).
\end{align*}
Also 
\begin{align*}
\lim_{\epsilon \to 0} \lim_{n \to \infty} \frac{1}{n}
\sum_{j=1}^{nk} f_\epsilon(d_j^{(n)})  =  \lim_{n \to
\infty} \sum_j \chi_K(d_j^{(n)}) = \lim_{n \to \infty} \frac{c_n(K)}{n}.  
\end{align*}
Applying Theorem \ref{thm_main} with $f=f_\epsilon$ and taking $\epsilon
\to 0$ we get \eqref{s5}. 
\end{proof}

\par We know that the map $d_j$ that takes a positive definite
matrix $B$ to its $j$th minimum symplectic eigenvalue
$d_j(B)$ is continuous \cite{bhatia-jain}. Since
$\theta \mapsto \wt{A}(\theta)$ is measurable on $[-\pi, \pi],$
the composite map $\theta \mapsto d_j(\theta) =
d_j(\wt{A}(\theta))$ is also measurable. 

\par Let $\mathcal{R}_j$ denote the essential range of the
map $d_j(\theta)$ and let $\mathcal{R}= \bigcup_{j=1}^{k}
\mathcal{R}_j.$ Since $\wt{A} \in L^\infty_{2k \times 2k},$
 the set $\mathcal{R}$ is compact.

\begin{lemma}\label{lemt1}
 For every $\wt{A}(\theta)$ in $\mathcal{R}(\wt{A})$ and $1\le j\le k,$
$d_j(A(\theta))$is in $\mathcal{R}_j.$
\end{lemma}

\begin{proof}
 Let $B=\wt{A}(\theta)$ be any element of $\mathcal{R}(\wt{A}).$
We show that $d_j(B)\in\mathcal{R}_j.$
Let $\epsilon>0.$
Since the map $d_j$ is continuous on positive definite matrices, we can find a $\delta>0$ such that
\begin{equation}
\|\wt{A}(t)-B\|<\delta\implies |d_j(\wt{A}(t))-d_j(B)|<\epsilon.\label{eqt2}
\end{equation}
By the definition of the essential range of $\wt{A},$
the set $S=\{t:\|\wt{A}(t)-B\|<\delta\}$
has positive measure.
Let $T$ be the set
$$T=\{t:|d_j(\wt{A}(t))-d_j(B)|<\epsilon\}.$$
By \eqref{eqt2}
we see that $S\subseteq T.$
Hence $T$ also has positive measure.
This shows that $d_j(B)\in\mathcal{R}_j.$
\end{proof}

\par For any subset $X$ of $\mathbb{R}$ let $\mathcal{B}(X, \delta)$ be its $\delta$-neighbourhood:
\[\mathcal{B}(X, \delta) = \{ x \in \mathbb{R}: |x -s | <
\delta \text{ for some } s \in X\}.\]
\par Let $\mathcal{D}_n$ be the set of 
symplectic eigenvalues of $T_n(\wt{A}).$ Let $\mathcal{D}=
\bigcup_n \mathcal{D}_n$
\begin{thm}\label{thm4.4}
The set $\mathcal{D}$ is dense in $\mathcal{R}.$ Further for each
$\delta>0$ let $X_\delta$ be the set 
\begin{equation}
X_\delta =[\mathfrak{m}_{\wt{A}}, \|\wt{A}\|] \backslash
\mathcal{B}(\mathcal{R},\delta).
\end{equation}
Then 
\begin{equation}
\lim\limits_{n \to \infty} \frac{c_n(X_\delta)}{n} =0.\label{eqt3}
\end{equation}
\end{thm} 
\begin{proof}
Since $\mathcal{R}$ is compact, we can apply Lemma \ref{ls5}
to get 
\begin{equation}\label{eq4.10}
\lim\limits_{n \to \infty} \frac{c_n(\mathcal{R})}{n} =
\frac{1}{2\pi} \sum\limits_{j=1}^k m (\{ \theta: d_j(\wt{A}(\theta)) \in
\mathcal{R}\}).
\end{equation}
\par Suppose $\mathcal{D}$ is not a dense subset of $\mathcal{R}$. Then there exist $j \in \{1, 2, \ldots,k\}$ $x \in \mathcal{R}_j$ and $\epsilon >0$ such that
$\mathcal{B}(x, \epsilon) \cap \mathcal{D} =
\emptyset.$ Since  $x \in \mathcal{R}_j$ the
set 
\[S =\{ t: |x- d_j(\wt{A}(t)) |< \epsilon\}\]
has a positive measure. Let $Y = \mathcal{R} \backslash
\mathcal{B}(x, \epsilon).$ The set $Y$ is compact, hence by \eqref{s5} we have
\begin{equation}
\lim\limits_{n\to\infty}\frac{c_n(Y)}{n}=\frac{1}{2\pi}\sum\limits_{j=1}^{k}m\{t:d_j(\wt{A}(t))\in Y\}.\label{eq4.11}
\end{equation}
Since $S$ has positive measure,
$$m\{t:d_j(\wt{A}(t))\in\mathcal{R}\}>m\{t:d_j(\wt{A}(t))\in Y\}.$$
This shows that the right hand side of \eqref{eq4.10} is strictly greater than the right hand side of \eqref{eq4.11}.
But since $\mathcal{D}\cap \mathcal{B}(x,\epsilon)=\emptyset,$
 $c_n(Y)=
c_n(\mathcal{R})$ for all $n.$
So,  the left hand sides of
 \eqref{eq4.10} and \eqref{eq4.11} are equal.
This is a contradiction. 
Hence $\mathcal{D}$ must be dense in $\mathcal{R}.$

\par The set $X_\delta$ is compact, and hence \eqref{s5} holds when $K=X_\delta.$
Since $X_\delta\cap\mathcal{R}_j=\emptyset$ for every $j,$
$$m\{\theta:d_j(\wt{A}(\theta))\in X_\delta\}=0.$$
This proves \eqref{eqt3}.
\end{proof}
\bibliographystyle{acm}
\bibliography{biblio}

\end{document}